\documentclass{daj}
\usepackage{amsfonts} 
\usepackage{amsmath} 
\usepackage{amsthm} 

\newtheorem{thm}{Theorem}
\newtheorem{lem}[thm]{Lemma}
\newtheorem{cor}[thm]{Corollary}

\dajAUTHORdetails{%
  title = {Proof of a Conjecture of Kleinberg-Sawin-Speyer}, 
  author = {Luke Pebody},
  plaintextauthor = {Luke Pebody},
}   

\dajEDITORdetails{%
   year={2018},
   volume={XX},
   number={13},
   received={4 February 2018},   
   revised={25 April 2018},    
   published={6 July 2018},  
   doi={10.19086/da.3733},       
}   

\begin{document}

\begin{frontmatter}[classification=text]


\author[lp]{Luke Pebody}

\begin{abstract}
In Croot, Lev and Pach's groundbreaking 
work~\cite{CrootLevPach}, the authors showed that a subset
of $\mathbb{Z}_4^{n}$ without an arithmetic
progression of length 4 must be of size at most
$3.1^n$. No prior upper bound of the form $(|G|-\epsilon)^n$
was known for the corresponding question in $G^n$ for any
abelian group $G$ containing elements of order greater than two.

Refining the technique in~\cite{CrootLevPach}, Ellenberg and Gijswijt~\cite{EllenbergGijswijt} showed that a subset of $\mathbb{Z}_3^{n}$ without an
arithmetic progression of length 3 must be of size at most 
$2.755^n$ (no prior upper bound of the form 
$(3-\epsilon)^n$ was known). They also provided, for any prime power $q$, a value 
$\lambda_q<q$ such that any subset of $\mathbb{Z}_q^{n}$ without an arithmetic progression of length 3 must be of size at
most $\lambda_q^n$.

Blasiak, Church, Cohn,
Grochow, Naslund, Sawin, and Umans~\cite{BlasiakEtAl} showed that the same 
bounds
apply to tri-coloured sum-free sets, which are triples 
$\{(a_i,b_i,c_i):a_i,b_i,c_i\in\mathbb{Z}_p^{n}\}$ with
$a_i+b_j+c_k=0$ if and only if $i=j=k$. 

Building on this work, an early version of a paper by 
Kleinberg, Sawin and 
Speyer~\cite{KleinbergSawinSpeyerEarly} gave a description of
a value $\mu_q$ such that for any
$\epsilon>0$, tri-coloured sum-free sets of size 
$e^{(\mu_q-\epsilon) n}$
exist in $\mathbb{Z}_q^{n}$ but, if $q$ is a prime power, for all
sufficiently large $n$, 
no tri-coloured sum-free sets 
of size $e^{\mu_q n}$ exist in $\mathbb{Z}_p^{n}$. The value of 
$\mu_q$ 
was
left open in the early version of this paper, 
but a conjecture was stated which would imply 
that 
$e^{\mu_q}=\lambda_q$, i.e. the Ellenberg-Gijswijt bound
is essentially tight for the tri-coloured sum-free set problem.

This note proves that conjecture and closes that gap. The 
conjecture of Kleinberg, Sawin and Speyer is true, and the 
Ellenberg-Gijswijt bound is essentially tight for the 
tri-coloured sum-free set problem.
\end{abstract}
\end{frontmatter}

\section{Introduction and Definitions}
Let $G$ be an abelian group. Define a 
{\em tri-coloured sum-free set in $G$} to be a 
collection of triples
$(a_i, b_i, c_i)$ in $G^3$ such that 
$a_i+b_j+c_k=0$ if and only if $i=j=k$. Write 
$\mathrm{sfs}(G)$
for the size of the largest tri-coloured sum-free set in $G$ and
$r_3(G)$ for the size of the largest set in $G$ with 
no
three-term arithmetic progression. If $X\subseteq G$ is a set 
with no three-term arithmetic progression, then
$\{(x, x, -2x):x\in X\}$ is a tri-coloured sum-free set, so 
$r_3(G)\le\mathrm{sfs}(G)$.

Until recently the question of whether $\lim_{n\to\infty}r_3(G^n)^{\frac1n}$ was less than $|G|$ was open for every
abelian $G$ containing elements of order greater than 2.

Croot, Lev and Pach~\cite{CrootLevPach} introduced a polynomial
method to show that there was indeed strict inequality for
$G=\mathbb{Z}_4$. Building on this, 
Ellenberg and Gijswijt~\cite{EllenbergGijswijt} showed there was
strict inequality for all cyclic groups with prime power order. They proved:
\begin{thm}\label{esStatement}
Let $p$ be a prime power, and let
$\theta_p$ denote the minimum value of \[\frac{1+\beta+\ldots+\beta^{p-1}}{\beta^{(p-1)/3}}.\] 

Then all sets in $\mathbb{Z}_p^n$ with no three-term arithmetic
progression are of size at most $C\theta_p^n$.
\end{thm}

Blasiak et al~\cite{BlasiakEtAl} showed that the same 
bounds apply to tri-coloured sum-free sets. Based on the result 
in this paper, Kleinberg, Sawin and 
Speyer~\cite{KleinbergSawinSpeyer} have shown this bound
is essentially tight for tri-coloured sum-free sets, showing that:

\begin{thm}\label{kssStatement}
For any integer $p$ (whether a prime power or not), with
$\theta_p$ defined as above,
for $n$ 
sufficiently large, there are tri-coloured sum-free sets in $\mathbb{Z}_p^n$ of size
at least 
\[\theta_p^ne^{-2\sqrt{(2\log\theta_p\log 2) n}-O_p(\log n)}.\]
\end{thm}

To give some motivation for the result in this paper, 
let $\Delta_p$ denote the set of distributions $\pi$
on $[p]=\{0, 1, 2, \ldots, p-1\}$ such that you can have
three variables $X_1, X_2$ and $X_3$ all distributed
according to $\pi$ with $X_1+X_2+X_3$ constant equal
to $p-1$.

Given a probability distribution $\pi$ on any finite set $S$,
define the {\em entropy} of $\pi$ to be 
\[\eta(\pi)=-\sum_{i\in S}\pi(i)\log\pi(i).\]

Then if we
define $\lambda_p=\max\{\eta(\pi):\pi\in\Delta_p\}$,
an early version of the paper by
Kleinberg, Sawin and Speyer~\cite{KleinbergSawinSpeyerEarly}
showed that:
\begin{thm}\label{kssStatement2}
Let $p$ be a positive integer, and define $\lambda_p$ as above. 
For any $n$, if $p$ is a prime power, then all tri-coloured sum-free sets in 
$\mathbb{Z}_p^n$ are of size at most 
$e^{\lambda_p n}$.

For all $p$ (whether a prime power or not), for $n$ 
sufficiently large, there are tri-coloured sum-free sets in $\mathbb{Z}_p^n$ of size
at least \[e^{\lambda_p n-2\sqrt{(2\lambda_p\log 2) n}-O_p(\log n)}.\]
\end{thm}

There is a natural comparison between this and the bound
of Ellenberg and Gijswijt~\cite{EllenbergGijswijt}. Clearly any distribution $\pi$ in
$\Delta_p$ has expected value $\frac{p-1}3$.
Let $\Delta'_p$ be the set of all probability
distributions on $[p]$ with expected value 
$\frac{p-1}3$. Then the problem of finding
$\lambda'_p=\max\{\eta(\pi):\pi\in\Delta'_p\}$ is considerably easier.

Let $\rho$ be the unique positive real number such that
the probability distribution
\[\psi_\rho(j)=\frac{\rho^j}{1+\rho+\ldots+\rho^{q-1}}:
0\le j\le q-1\]
has expected value $\frac{p-1}3$. Then $\psi_\rho$ is
the unique distribution in $\Delta'_p$ with maximal entropy. Further, this maximal entropy  is precisely $\log\theta_p$
where $\theta_p$ is as defined in the statement 
of Theorem~\ref{esStatement}.

In the early version of their paper~\cite{KleinbergSawinSpeyerEarly},
Kleinberg, Sawin and Speyer
conjectured that this distribution $\psi_\rho$ is in 
$\Delta_p$, and so the upper bounds in 
Theorem~\ref{kssStatement2} and 
Theorem~\ref{esStatement} are the same. The purpose of this note is to prove this conjecture. 

Say that probability distributions $\pi_1, \pi_2, \pi_3$ 
are {\em compatible} if 
we can choose dependent random variables
$X_1, X_2, X_3$ such that $X_i$ has distribution 
$\pi_i$ for each $i$ and $X_1+X_2+X_3$ is
constant, and say that a discrete non-negative integer-valued probability distribution $\pi$ is
{\em decreasing} if $\pi(0)\ge\pi(1)\ge\ldots$.

We will prove the following:
\begin{thm}\label{T:main}
If three discrete non-negative
integer-valued probability distributions are decreasing, 
only take values in $[p]=\{0, 1, \ldots, p-1\}$ and 
have expected values summing to 
$p-1$ then they are compatible.
\end{thm} 

Since $\psi_\rho$ is clearly decreasing (as $\rho<1$), it will quickly 
follow that $\psi_\rho$ is in $\Delta_p$
thereby completing the proof of Kleinberg, Sawin and 
Speyer~\cite{KleinbergSawinSpeyer} that the 
Ellenberg-Gijswijt bound is essentially tight for the 
tri-coloured sum-free set problem.

\section{Simple Distributions}
In this section we will use a convexity argument to 
show that it will be sufficient to prove Theorem~\ref{T:main}
if we can prove it for fairly simple distributions.

\begin{lem}\label{L:simpleDistributions}
Given any finite collection $\phi_1, \phi_2, \ldots, \phi_n$ of probability distributions on any finite
set $S$, we can simultaneously express the distributions as
 a non-negative linear 
combinations $\phi_i=\Sigma_jp_j\phi_{i,j}$ of distributions which satisfy:
\begin{enumerate}
\item For each j, the sum of the expectations of the $\phi_{i,j}$ is the same as that of the $\phi_i$
\item For each j, at least $n-1$ of the $\phi_{i,j}$ are constant and the other (if it is not constant) takes two values.
\end{enumerate}
\end{lem}

\begin{proof}
We proceed by induction on the sum of the support sizes of the
distributions. Suppose the support sizes sum to $m$, and the result is true whenever the
support sizes sum to less than $m$. For $1\le i\le n$, let 
$\mathrm{min}_i$ denote the smallest element of the support of $\phi_i$
and $\mathrm{max}_i$ the largest.

Clearly $\Sigma_i\mathrm{min}_i\le
\Sigma_i\mathbb{E}(\phi_i)\le
\Sigma_i\mathrm{max}_i$. Changing one of the parameters at a 
time over
from $\mathrm{min}_i$ to $\mathrm{max}_i$ shows there exists a $k$ such that 
\begin{align*}
\Sigma_{i<k}\mathrm{max}_i+\mathrm{min}_k+\Sigma_{i>k}\mathrm{min}_i&\le
\mathbb{E}(\phi_1)+\mathbb{E}(\phi_2)+\ldots+\mathbb{E}(\phi_n)\\
&\le
\Sigma_{i<k}\mathrm{max}_i+\mathrm{max}_k+\Sigma_{i>k}\mathrm{min}_i.
\end{align*}

Thus if we let $\phi_{i,1}$ be constant equal to $\mathrm{max}_i$ for $i<k$ and
constant equal to $\mathrm{min}_i$ for $i>k$ and to either be equal to $\mathrm{min}_k$
or $\mathrm{max}_k$ for $i=k$, we can rig up the probabilities to have 
$\Sigma_i\mathbb{E}(\phi_i)=
\Sigma_i\mathbb{E}(\phi_{i,1}).$

Now if $\phi_i=\phi_{i,1}$ for all $i$, we are done. Otherwise let 
$p_1$ be the smallest value of $\phi_{i}(t)/\phi_{i,1}(t)$
(ranging over $i, t$ where the denominator is non-zero)
and then
$\phi'_i=\frac{1}{1-p_1}({\phi_i(k)-p_1\phi_{i,1}(k)})$ are 
distributions with the same sum of expectations but with at least one
support size strictly smaller, so we can use the induction hypothesis to 
generate the remaining distributions.
\end{proof}

Given an integer-valued distribution $\phi$, denote by 
$\widehat{\phi}$ the distribution you get by first choosing a value 
$i$ according to $\phi$ and then choosing a value $j$ uniformly from
$\{0,1,\ldots,i\}$. Clearly $\widehat{\phi}$ is decreasing and any
decreasing distribution is of this form. Furthermore, clearly this operator
respects linear combinations.

We will apply this operator to the distributions in 
Lemma~\ref{L:simpleDistributions}. It will be useful to define
the distribution $U_k$ 
for each non-negative integer $k$ as the uniform distribution on 
$\{0, 1, \ldots, k\}$ 
and for positive integers $k, l$ and real $x$ with $k<x< l$ the distribution
$V_{k,l,x}$ as a linear combination of $\frac{l-x}{l-k}$ weight of $U_k$ and 
$\frac{x-k}{l-k}$ weight of $U_l$, weights so chosen as to make the expectation
equal to $\frac{x}{2}$.

\begin{cor}\label{C:simpleDistributions}
Suppose that the positive integer $p$ satisfies that for all 
non-negative integers 
$0\le x_1, x_2, x_3<p$ with 
$x_1+x_2+x_3=2(p-1)$, $U_{x_1}, U_{x_2}$ and $U_{x_3}$
are compatible.

Suppose also that for all non-negative integers $0\le x_1, x_2,
y, z<p$ with 
$x_1+x_2+y<2(p-1)<x_1+x_2+z$, $U_{x_1}, U_{x_2}$ and 
$V_{y, z, 2k-x_1-x_2}$ are compatible.

Then any triple of decreasing distributions on $[p]$
with expected value equal to $p-1$ are compatible.
\end{cor}

\begin{proof}
Take such a triple $\{\phi_1, \phi_2, \phi_3\}$.
Then as noted above, the $\phi_i$ can be written in the form 
$\widehat{\psi_i}$
where the $\psi_i$ for some distribution $\psi_i$ bounded above by $k$ with 
expected values summing to $2k$.

By Lemma~\ref{L:simpleDistributions}, $\psi_1, \psi_2, \psi_3$
can be expressed as a non-negative 
linear combination of distributions with the same sum
of expected values where two are constant 
and the third takes at most two values.

Reapplying the operator it follows that $\phi_1, \phi_2, 
\phi_3$ can
be expressed as a non-negative
 linear combination of distributions with the same sum of
expected values where two are of the form $U_k$ 
and the third is either of the form $U_k$ or of the form
$V_{k,l,x}$. Therefore if the assumptions of this Corollary hold, they are the
linear combination of compatible distributions.

Since the non-negative linear combination of 
compatible distributions yields 
compatible
distributions, it follows the original collection is compatible.
\end{proof}

\section{The Induction Step}
Our proof of Theorem~\ref{T:main} will be by induction on $k$.
We deal first with the case where at most one of the distributions has support 
including $k-1$.

\begin{lem}\label{L:inductionstep}
If $\pi_1, \pi_2, \pi_3$ are decreasing distributions on $[p]$ with
$\mathbb{E}(\pi_1)+\mathbb{E}(\pi_2)+\mathbb{E}(\pi_3)=p-1$, then for all integers $t$ in the range $0\le t<p$,
$\mathbb{P}(\pi_1=0)\le\mathbb{P}(\pi_2>t)+\mathbb{P}(\pi_3\ge(p-1)-t)$.
\end{lem}

\begin{proof}
If $\pi$ is a decreasing distribution on $[p]$ then for any integer
$0\le t<p$, the expected value of $\pi$ given $\pi>t$ is at most $\frac{t+p}2$
and the expected value of $\pi$ given $\pi\le t$ is at most $\frac t2$.

Thus the expected value of $\pi$ is at most $\frac t2+\frac p2\mathbb{P}(\pi>t)$.

Therefore
\begin{align*}
p-1&=\mathbb{E}(\pi_1)+\mathbb{E}(\pi_2)+\mathbb{E}(\pi_3)\\
&\le\frac{0+t+(p-2)-t}2+\frac p2(
\mathbb{P}(\pi_1>0)+\mathbb{P}(\pi_2>t)+\mathbb{P}(\pi_3>p-2-t))\\
&\le\frac{p-2}2+\frac p2(
\mathbb{P}(\pi_1>0)+\mathbb{P}(\pi_2>t)+\mathbb{P}(\pi_3>p-2-t))
\end{align*}

which simplifies to $1\le\mathbb{P}(\pi_1>0)+\mathbb{P}(\pi_2>t)+\mathbb{P}(\pi_3>p-2-t)$.
Subtracting $\mathbb{P}(\pi_1>0)$ from both sides gives the required inequality.
\end{proof}

This allows us to handle a lot of cases by induction.
\begin{cor}\label{C:inductionstep}
Suppose that for all decreasing distributions $\pi_1, \pi_2, \pi_3$ on $[p-1]$
with $\mathbb{E}(\pi_1)+\mathbb{E}(\pi_2)+\mathbb{E}(\pi_3)=p-2$, $\pi_1, \pi_2$ and $\pi_3$ are
compatible.

Then for all decreasing distributions $\pi_1, \pi_2, \pi_3$ on $[p]$ with
$\mathbb{E}(\pi_1)+\mathbb{E}(\pi_2)+\mathbb{E}(\pi_3)=p-1$ with $\pi_2(p-1)=\pi_3(p-1)=0$,
$\pi_1, \pi_2$ and $\pi_3$ are compatible.
\end{cor}

\begin{proof}
Clearly $\pi_2(0)>\pi_3(p-1)$ and $\pi_2(p-1)<\pi_3(0)$, so there must exist
an integer $t$ in the range $0\le t<p$ such that $\pi_2(k)>\pi_3((p-1)-k)$ for $k\le t$ and
$\pi_2(k)\le\pi_3((p-1)-k)$ for $k>t$. 

This means that 
\[\Sigma_{k\le t}\min(\pi_2(k), \pi_3((p-1)-k))=\Sigma_{k\le t}\pi_3((p-1)-k)=\mathbb{P}(\pi_3\ge p-1-t)\]
and
\[\Sigma_{k>t}\min(\pi_2(k), \pi_3((p-1)-k))=\Sigma_{k>t}\pi_2(k)=
\mathbb{P}(\pi_2>t).\]

By Lemma~\ref{L:inductionstep}, it follows that
$\mathbb{P}(\pi_1=0)\le\Sigma_k\min(\pi_2(k), \pi_3((p-1)-k)).$
Thus there exists a real number $x$ such that 
\[\mathbb{P}(\pi_1=0)=\Sigma_k\min(\pi_2(k), \pi_3((p-1)-k), x).\]

For $0\le k\le p-1$, let $f(k)=\min(\pi_2(k), \pi_3((p-1)-k), x)$. Since 
$\min(\pi_2(k), x)$ is non-increasing, if $f(k)<f(k+1)$, 
$f(k)=\pi_3((p-1)-k)$. Similarly, if $f(k)>f(k+1)$, $f(k+1)=\pi_2(k+1)$.

Therefore if we define $g_2(k)=\pi_2(k)-f(k)$, either $f(k)\le f(k+1)$, 
and hence $g_2(k)\ge g_2(k+1)$, or $f(k)>f(k+1)$, in which case 
$g_2(k)\ge 0=g_2(k+1)$. Thus $g_2$ is non-increasing. Similarly, so is
$g_3(k)=\pi_3(k)-f((p-1)-k)$.

If we start defining our dependent random variables by
saying that with probability $f(k)$, $X_1=0$, 
$X_2=k$ and $X_3=(p-1)-k$, 
the remaining distribution on $\pi_1$ is 
non-increasing on $\{1, \ldots, p-1\}$ and the 
remaining distributions on 
$\pi_2$ and $\pi_3$ are non-increasing on $\{0, \ldots, p-2\}$.

Thus if we subtract 1 from the remaining distribution on $\pi_1$, we are
left with decreasing distributions on $\{0, \ldots, p-2\}$ with expected
values that sum to $p-2$, which are compatible. It follows that $\pi_1$, 
$\pi_2$ and $\pi_3$ must be compatible.
\end{proof}

\section{The remaining case}
Corollary~\ref{C:inductionstep} is enough to cover almost all cases of the assumptions of
Corollary~\ref{C:simpleDistributions}. In this section we will deal with the
remaining case.

\begin{lem}\label{L:lastcase}
For any non-negative integers $m$ and $n$ one can pair distributions 
$X$ and $Y$ that are uniform on $\{0, 1, \ldots, m\}$ and
$\{0, 1, \ldots, n\}$ such that $X+Y$ is uniform on
$\{0, 1, \ldots, m+n\}$.
\end{lem}

\begin{proof}
Let $A_1, \ldots, A_m, B_1, \ldots, B_n$ and $C$ be 
independent identically distributed random variables with continuous
distributions. Denote $|{i:A_i<C}|$ by $X$ and $|{j:B_j<C}|$
by $Y$. Clearly 
$X$ is uniform on $\{0, 1, \ldots, m\}$, $Y$ is uniform on
$\{0, 1, \ldots, n\}$ and $X+Y$ is uniform on $\{0, 1, \ldots, m+n\}$.
\end{proof}

Combining two instances of this lemma together gives the following.

\begin{cor}\label{C:lastcase}
For any integers $i$ and $j$ with $i>0$, $j\ge 0$ 
and $i+j<p-1$, the distributions
$U_i$, $U_{p-1}$ and $V_{j,p-1,p-1-i}$ are compatible.
\end{cor}

\begin{proof}
Firstly flip a coin which lands on heads with probability $\frac{i+j+1}{p}$.

If the coin comes up heads, use Lemma~\ref{L:lastcase} to sample random
variables $X$ and $Y$ where $X$ is uniform on $\{0, 1, \ldots, i\}$,
$Y$ is uniform on $\{0, 1, \ldots, j\}$ and $X+Y$ is uniform on 
$\{0, 1, \ldots, i+j\}$. Then let $X_1=X, X_2=(p-1)-X-Y$ 
and $X_3=Y$.

If the coin comes up tails, use Lemma~\ref{L:lastcase} to sample random
variables $X'$ and $Y'$ where $X'$ is uniform on $\{0, 1, \ldots, i\}$,
$Y'$ is uniform on $\{0, 1, \ldots, p-(i+j+2)\}$ 
and $X'+Y'$ is uniform on
$\{0, 1, \ldots, p-(j+2)\}$. Then let $X_1=X', X_2=Y'$
and $X_3=(p-1)-X'-Y'$.

$X_1$ is uniform on $\{0, 1, \ldots, i\}$ regardless of the coin-flip.
$X_2$ is uniform on $\{0, 1, \ldots, p-(i+j+2)\}$ with
probability $\frac{p-(i+j+1)}{p}$ and uniform on 
$\{p-(i+j+1),\ldots,p-1\}$ with probability $\frac{i+j+1}{p}$, so is uniform
on $\{0, \ldots, p-1\}$.

Finally $X_3$ is a combination of uniform on $\{0, 1, \ldots, j\}$
and uniform on $\{j+1, \ldots, p-1\}$ with mean $\frac{p-1-i}2$, so it must be 
$V_{j,p-1,p-1-i}$.
\end{proof}

\section{Putting it all together}
We now have all of the ingredients for the proof of Theorem~\ref{T:main}.

\begin{proof}[Proof of Theorem~\ref{T:main}]
We proceed by induction on $p$.

If $p=1$, $\pi_1$, $\pi_2$ and $\pi_3$ are constant equal to 0 and are trivially compatible.

Now suppose that Theorem~\ref{T:main} is true for $p-1$, and we will 
attempt to prove it for $p$. By Corollary~\ref{C:simpleDistributions},
it is sufficient to prove that $U_k$, $U_l$ and $U_m$ are 
compatible when 
$0\le k, l, m<p$ and $k+l+m=2(p-1)$ and $U_k, U_l$ and 
$V_{m,n,2(p-1)-k-l}$ are compatible when $0\le k, l, m, n<p$ with
$k+l+m<2(p-1)<k+l+n$.

By Corollary~\ref{C:inductionstep}, the induction hypothesis implies
it is true for all such $U_k, U_l$ and $U_m$ if two of $k, l$ and $m$
are less than $p-1$ and
for all such $U_k, U_l$ and $V_{m,n,2(p-1)-k-l}$ if two of $k, l$ and 
$n$ are less than $p-1$.

So, for $U_k, U_l$ and $U_m$ we can assume that at least two of 
$k, l$ and $m$ are equal to $p-1$. Since $k+l+m=2(p-1)$, it follows
that two are uniform on $[p]$ and the other is constant equal to 0.
Then these are compatible: let $X_1=0$, $X_2$ be uniform
on $[p]$ and let $X_3=(p-1)-X_2$.

Similarly, for $U_k, U_l$ and $V_{m,n,2(p-1)-k-l}$, we can assume that 
at least two of $k, l$ and $n$ are equal to $p-1$. Since $k+l+m<2(p-1)$,
it follows that $n$ must equal $p-1$ and one of the others must be
equal to $p-1$, so they are of the form stated in Corollary~\ref{C:lastcase}.
\end{proof}

Finally, since the distribution $\psi_\rho$ is clearly 
decreasing (as $\mathbb{E}\psi_\rho=\frac{p-1}3<\frac{p-1}2$, and hence $\rho<1$), it follows from Theorem~\ref{T:main} that
$\psi_\rho$ is compatible with two copies of itself,
which means that $\psi_\rho\in\Delta_p$, completing
the proof by Kleinberg-Sawin-Speyer that
the Ellenberg-Gijswijt bound
is essentially tight for tri-coloured sum-free sets.

\bibliographystyle{amsplain}


\begin{dajauthors}
\begin{authorinfo}[pgom]
  Luke Pebody\\
  Rokos Capital Management\\
  23 Savile Row\\
  London, United Kingdom\\
  luke\imageat{}pebody\imagedot{}org 
\end{authorinfo}
\end{dajauthors}

\end{document}